\documentclass[11pt]{amsart}

\usepackage{graphicx}

\newtheorem{thm}{Theorem}[section]

\newcommand{\obar}[1]{\ensuremath{\overline{#1}}}

\newcommand{\B}[2]{\ensuremath{B_{\overline{#1}#2}}}

\newcommand{\CC}[1]{\ensuremath{\mathbf{C}^{#1}}}
\newcommand{\CP}[1]{\ensuremath{\mathbf{CP}^{#1}}}
\newcommand{\RP}[1]{\ensuremath{\mathbf{RP}^{#1}}}
\newcommand{\R}[2]{\ensuremath{\mathcal{R}_{\overline{#1}#2}}}

\newcommand{\XX}{\ensuremath{\mathcal{X}}}

\newcommand{\C}[1]{\ensuremath{\mathcal{C}_{#1}}}
\newcommand{\Cb}[1]{\ensuremath{\mathcal{C}_{\obar{#1}}}}

\newcommand{\Sym}[1]{\ensuremath{\mathcal{S}_{#1}}}
\newcommand{\Alt}[1]{\ensuremath{\mathcal{A}_{#1}}}

\newcommand{\St}[1]{\ensuremath{\mathrm{St}{(#1)}}}
\newcommand{\I}[1]{\ensuremath{\mathcal{I}_{#1}}}
\newcommand{\Ib}[1]{\ensuremath{\mathcal{I}_{\obar{#1}}}}
\newcommand{\Oct}[1]{\ensuremath{\mathcal{O}_{#1}}}
\newcommand{\Octb}[1]{\ensuremath{\mathcal{O}_{\obar{#1}}}}

\newcommand{\V}{\ensuremath{\mathcal{V}}}
\newcommand{\Vh}[1]{\ensuremath{\widehat{\mathcal{V}}_{#1}}}

\newcommand{\pQ}[2]{\ensuremath{p_{\obar{#1}#2}}}
\newcommand{\pV}[2]{\ensuremath{p_{\obar{#1}#2}}}
\newcommand{\pT}[2]{\ensuremath{p_{#1* #2}}}
\newcommand{\pTb}[2]{\ensuremath{p_{\obar{#1* #2}}}}

\newcommand{\F}[1]{\ensuremath{F_{#1}}}

\title[New light on the sextic]{New light on solving the sextic by iteration:\\
An algorithm using reliable dynamics}

\author{Scott Crass}
\address{Mathematics Department\\
CSULB\\
Long Beach, CA  90840-1001}
\email{scrass@csulb.edu}

\begin{document}

\maketitle

\section{Overview: Light that is new}

In the mid 1990s the author developed an algorithm that approximates the roots of one-variable polynomials in degree six.  The method employed the dynamical properties of a special map
$h$ on \CP{2}.  A distinguishing feature of this approach to equation-solving is that $h$ is \emph{equivariant}; it respects the symmetries of the \emph{Valentiner group} \V---a collection of projective transformations isomorphic to the alternating group \Alt{6}.  Recall that \Alt{6} is the Galois group of a generic sextic after extracting the square-root of its discriminant.  Being equivariant, $h$ sends a group orbit to a group orbit.  In algebraic terms, $h$ commutes with each element of \V.  While the geometric behavior of $h$ is elegant, its crucial global dynamical properties were and remain unproven.  In particular, it was conjectured that a certain Valentiner orbit of $72$ superattracting periodic points has basins of attraction that form a dense subset of \CP{2}.

Recent results of work on a certain family of \emph{complex reflection groups} \cite{crit_fin} as well as on an iterative solution to certain seventh-degree equations \cite{heptic168} led to a substantial dynamical advance: the discovery of a new ``\V-equivariant" $g$ whose dynamics is known to have the desired global properties.  A reflection group---a class to which the Valentiner action belongs---is generated by transformations each of which pointwise fixes a \emph{reflection hyperplane}.  \cite{st}  Indeed, $g$ fits into a pattern seen in a number of reflection groups.  Namely, it both fixes and is critical on the reflection hyperplanes---lines, in this case.  This property of being \emph{critically finite} strongly influences the map's dynamics.  So much so that proofs for its global dynamical behavior are no longer elusive.

The following discussion 1) examines the geometry and algebra of the Valentiner group that is essential for constructing and analyzing $g$, 2) explicitly computes $g$ and establishes its crucial properties, and 3) harnesses $g$'s dynamics to a \emph{reliable} iterative algorithm for solving the sextic.  `Reliable' means that the procedure succeeds in producing a root on a open dense set's worth of polynomials and initial conditions.

Computational results are due to \emph{Mathematica}.  Dynamical plots were generated by \emph{Julia} \cite{julia} and \emph{Dynamics 2} \cite{dyn2}. 
\section{Elements of the Valentiner group: Mirrors, polynomials, and maps}

The treatment here is minimal, pointing out only those features relevant to the construction of $g$, the special \V-equivariant in degree $31$.  More thoroughgoing examinations of the Valentiner group appear in \cite{note} and, especially, \cite{sextic}.

\subsection{Lifting the action}

Working with the Valentiner group computationally requires having a matrix representation.  The \emph{minimal} lift of \V\ to \CC{3} occurs in a one-to-three fashion yielding the group \Vh{3} where $\det{T}=1$ for all $T\in \Vh{3}$.  Extending to
$$\Vh{6}=\Vh{3} \cup -\Vh{3}$$
gives a \emph{complex reflection group}.  (See \cite{st}.)

\subsection{Basic  geometry}

Each of the $45$ transposition-pairs $(ab)(cd)$ in \Alt{6} corresponds to a \emph{complex reflection} on \CP{2} that belongs to \V.
Such a transformation fixes a \CP{1} pointwise---a complex mirror---as well as a point not on the mirror.  For convenience, call these objects ``$45$-lines" and ``$45$-points" respectively.

Among alternating groups, \Alt{6} is distinguished by having \emph{two} systems of six versions of \Alt{5}---the six things that \Alt{6} permutes under conjugation.  The Valentiner counterpart \I{a} or \Ib{c} ($a,c=1\dots 6 $) to an \Alt{5} subgroup preserves a quadratic form $C_{a}$ or $C_{\bar{c}}$ and, hence, the associated conics
$$\C{a}=\{ C_a=0\} \quad \text{or} \quad  \Cb{c}=\{ C_{\bar{c}}=0\}.$$
Each conic is, by parametrization, a \CP{1} thereby endowed with an icosahedral structure.

Along with the $45$-lines, the icosahedral conics are fundamental to Valentiner geometry. Another basic piece of the Valentiner structure concerns the special \V-orbits where $45$-lines intersect each other as well as the conics.  Five of these lines meet at a $36$-point, four at a $45$-point,  and three at a $60$-point associated with either the ``barred" or ``unbarred" family of conics.  Special points and certain icosahedral groups have a taxonomic association that gives rise to a convenient terminology.  A pair of barred or unbarred icosahedral conics determines an octahedral structure, namely,
$$
\Octb{ac}=\St{\Cb{a}\cup \Cb{c}}\simeq \Sym{4} \quad \text{and}\quad \Oct{bd}=\St{\C{b}\cup\C{d}}\simeq \Sym{4}
$$
are octahedral groups by virtue of their preserving the conics
$$
\Cb{ac}=\{C_{\obar{a}}+C_{\obar{c}}=0\} \quad \text{and}\quad
 \C{bd}=\{C_{b}+C_{d}=0\}.
$$
Here, $\mathrm{St}(X)$ denotes the stabilizer of $X$.  Table \ref{tab:specPts} summarizes the special structure.

\begin{table}[ht]
\caption{Special points}
\label{tab:specPts}
\centering

\small

\fbox{
\begin{tabular}{cccc}

Special point&Descriptor&Stabilizer\\[-10pt]
\hrulefill&\hrulefill&\hrulefill\\

$36$-point& \pQ{a}{b}&$\Ib{a} \cap \I{b}\simeq D_5$\\
$45$-point& \pV{ac}{bd}&$\Octb{ac}\cap \Oct{bd}\simeq D _4$\\
$60$-point& \pT{a}{bcd}&$\I{e} \cap \I{f} \cap \I{a}\simeq \Sym{3}$&
$\{e,f\}=\{1,\dots,6\}-\{a,b,c,d\}$\\
\obar{60}-point& \pTb{a}{bcd}&$\Ib{e} \cap \Ib{f} \cap \Ib{a}\simeq \Sym{3}$&$\{e,f\}=\{1,\dots,6\}-\{a,b,c,d\}$\\

\end{tabular}
}

\end{table}

It turns out that there is an antiholomorphic involution $B$  that exchanges the two systems of icosahedral conics.  A choice of coordinates yields  $B(y)=\obar{y}$ so that $B$ acts as a kind of reflection through an \RP{2}---the points with real $y$ coordinates.  Indeed, there are $36$ such antilinear maps $\B{a}{b}$ in $ B \V$---one for each barred-unbarred pair---that form a set of generators for the extension group
$$\obar{\V}=\V \cup B \V.$$
Of course, each \B{a}{b} fixes an \RP{2} mirror \R{a}{b}.

\subsection{Generating invariants and equivariants}

Under \Vh{6}, the forms $C_a$ (or $C_{\obar{c}}$) are permuted up to a factor of a cube root of unity.  Thus, the $C_{a}^3$ (or $C_{\obar{c}}^3$) undergo \emph{simple} permutation; that is,
$$
\text{for all T}\ \in \Vh{6}\ \text{and}\ c \in \{1,\dots,6\},\
C_a^3\circ T = C_{\tau(a)}^3\ \text{where}\ \tau(a)\in \{1,\dots,6\}.
$$
Accordingly,
$$
F= \alpha\,\sum_{a=1}^6 C_a^3= \obar{\alpha}\,\sum_{c=1}^6 C_{\obar{c}}^3
$$
is \Vh{6}-invariant ($\alpha$ is a simplifying constant).  Moreover, $F$ is the \Vh{6} invariant of least degree and, from it, we can derive two other \Vh{6} forms
$$\Phi=\det{H(F)} \qquad \Psi=\det{BH(F,\Phi)}$$
where $H(F)$ is the hessian of $F$ and $BH(F,\Phi)$ is the \emph{bordered hessian} of $F$ and $\Phi$.  Furthermore, the algebraically independent set $\{F,\Phi,\Psi\}$ generates the ring of \Vh{6}-invariant polynomials $\CC{}[y]^{\Vh{6}}$.

Another form that plays an important role is the jacobian determinant
$$X=\det{J(F,\Phi, \Psi)}$$
which factors into $45$ linear forms that define the $45$-lines.  However, while $X$ is \Vh{3}-invariant, it enjoys only \emph{relative} \Vh{6} invariance.  Specifically, $$X\circ T=-X \quad \text{for} \quad T\in \Vh{6}-\Vh{3}$$
so that $X^2$ is \Vh{6}-invariant and thereby expressible in terms of $F$, $\Phi$, and $\Psi$.

We can form \Vh{6}-equivariant maps by manipulation of the basic invariants:
$$
\psi=\nabla F \times \nabla \Phi \qquad
  \phi=\nabla F \times \nabla \Psi \qquad
  f=\nabla \Phi \times \nabla \Psi
  $$
where $\nabla$ indicates a formal gradient operator.  To obtain a set of generators for the \Vh{3} equivariants as a module over $\CC{}[y]^{\Vh{6}}$, include three more maps: the identity,
$$
h=\frac{h_{64}}{X}, \quad \text{and} \quad k=\frac{k_{70}}{X}
$$
where $h_{64}$ and $k_{70}$ are \Vh{6} equivariants of degree $64$ and $70$---expressed in terms of $\psi$, $\phi$, and $f$---that vanish on $\{X=0\}$.   After division by $X$, the result is a degree-$19$ map $h$ (``$19$-map" for short) and $25$-map $k$. 
\section{The $31$-map: Algebra, geometry, and dynamics}

\subsection{Algebraic derivation and geometric properties}

Consider the family of $76$-maps,
\begin{align*}
f_{76}=&\ (a_1\,F^{10} + a_2\,F^8 \Phi + \dots + a_{10}\,\Psi^2)\,\psi\\
 &+ (b_1\,F^7 + \dots + b_6\,\Phi \Psi)\,\phi
+(c_1\,F^6 + \dots + c_5\,\F \Psi)\,f.
\end{align*}
Utilizing the technique that produced the maps $h$ and $k$, determine coefficients so that $X$ divides each coordinate polynomial of $f_{76}$ and, thereby, obtain a seven-parameter family of $31$-maps
$$f_{31}=\frac{f_{76}}{X}.$$
Next, by solving a nonlinear system of equations, force $f_{31}$ to be critical on a $45$-line and, thus, on all such lines.  The resulting map $g$ has a number of special geometric and dynamical properties that make it particularly well-suited for residing at the core of a sextic-solving algorithm. First, consider $g$'s algebraic form.  For later purposes, having its degree-$76$ expression is convenient:
\small \begin{align*}
&X\cdot g=-162\,(67\,F^6+812\,\Phi F^4+699\,\Phi^2
   F^2-741\,\Psi F-150\,\Phi^3)f \\
   &- 18\,(19\,F^7-1016\,\Phi F^5-13457\,\Phi^2
   F^3+18\,\Psi F^2-11970\,\Phi^3 F+13500\,\Phi
   \Psi)\phi \\
   &+ (-76\,F^{10}-564\,\Phi
   F^8+5244\,\Phi^2 F^6-48672\,\Psi F^5+66028\,\Phi^3 F^4\\
   &-591336\,\Phi \Psi F^3+232848\,\Phi^4 F^2-517752\,\Phi^2 \Psi F+70200\,\Phi^5+577125\,\Psi^2)\psi.
\end{align*} \normalsize

Since $g$ is critical on $\XX=\{X=0\}$, $X$ divides $g$'s critical polynomial
$$C_g=\det{J_{\widehat{g}}}$$
where $J_{\widehat{g}}$ is the Jacobian matrix of a lift $\widehat{g}$ of $g$ to \CC{3}.  Moreover, the degree of $C_g$ is $90=3 (31-1)$ so that $X^2$ also divides $C_g$.  That is, for some non-zero $\beta\in \CC{}$,
$$C_g=\beta\,X^2.$$
Adapting arguments from \cite{crit_fin} establishes some basic properties of $g$.

\begin{thm} \label{thm:g_holom}
In the family of \V-equivariant $31$-maps, $g$ is the unique holomorphic member that is doubly-critical on \XX.
\end{thm}

\begin{proof}
Let $\widehat{g}$ be a lift of $g$ to \CC{3} and make a linear change of coordinates
$$(y_1,y_2,y_3)=T(x_1,x_2,x_3)$$
so that $L=\{y_3=0\}$.  Denoting the map in $y$ also by $g$,
$$\widehat{g}(y_1,y_2,0) = (s(y_1,y_2),t(y_1,y_2),0).$$
If $\widehat{g}(a)=(0,0,0)$, the homogeneous components $s$ and $t$ have a common factor which, contrary to fact, requires the resultant of $s$ and $t$ to vanish.
\end{proof}

\subsection{Dynamical behavior}

 The central dynamical object here is the invariant Fatou set $\mathcal{F}_g$ where $\{g^k\}$ is a normal family.  Its complement is the Julia set $\mathcal{J}_g.$

Since a $45$-line $L$ is fixed pointwise by an element of \V, a holomorphic \V-equivariant fixes $L$ setwise.  Accordingly, $g$ is \emph{strictly critically finite} (see \cite{fs}).   An intersection of $45$-lines is a fixed point that is superattracting both along and away from the line.  Elsewhere on a $45$-line, $g$ is critical in the off-line direction so that the line is a superattracting set.   Following \cite{crit_fin} and \cite{ueno}, we can see that $g$ possesses two global properties---what `reliable dynamics' means.  The arguments' details can be found in these sources.

First, a computation that expresses $g$ when restricted to a $45$-line $L$: in symmetrical coordinates where the four $45$-points on $L$ are
$$\bigl\{60^{-1/4}(\pm 1 \pm i)\bigr\},$$
\small
$$g|_L(z)= \frac{p(z)}{q(z)}$$
where
\small
\begin{align*}
p(z)=&\
-15 (60381703125 z^{30}-87218015625
   z^{28}+219610490625 z^{26}\\
&-169714828125
   z^{24}+21618140625 z^{22}-52377193125
   z^{20}+10578808125 z^{18}\\
&-14597849625
   z^{16}-555500625 z^{14}-105496875
   z^{12}+6315475 z^{10}\\
&-2034375 z^8-220565 z^6-5735 z^4-465 z^2-3)\\
q(z)=&\
   z (115330078125 z^{30}-1191744140625
   z^{28}+979878515625 z^{26}\\
&-2512373203125
   z^{24}+1544853515625 z^{22}+319720921875
   z^{20}\\
&+356051953125 z^{18}-124987640625
   z^{16}+218967744375 z^{14}\\
&+10578808125
   z^{12}+3491812875 z^{10}+96080625
   z^8+50285875 z^6\\
&+4337985 z^4+114855
   z^2+5301).
\end{align*}
\normalsize
The remaining intersections of $45$-lines---$36$, $60$, and $\obar{60}$-points and $L$ are identified in Figure~\ref{fig:crit_L45}.

\begin{figure}[ht]

\resizebox{3in}{!}{\includegraphics{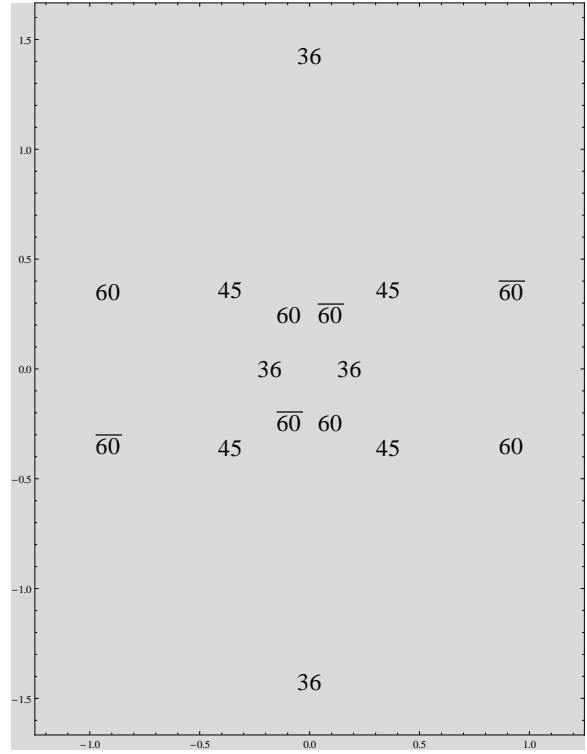}}

\caption{Critical points for $g$ restricted to a $45$-line}

\label{fig:crit_L45}

\end{figure}

Now, we consider the content of $g$'s Fatou set.

\begin{thm} \label{thm:fatou}
The $36$, $45$, $60$, and $\obar{60}$-points are the only superattracting fixed points of $g$.  Moreover, their basins exhaust the Fatou set $\mathcal{F}_g$ of $g$.
\end{thm}

\begin{proof}
Since the $45$-lines don't blow down to points, their intersections are fixed and critical in all directions.  Conversely, if $a$ is critical in all directions, it belongs to a $45$-line $L$ and is a critical point under $g_L=g|_L$.  But, the sixteen points of intersection between $L$ and the remaining $45$-lines account for all such points.  To see this, compute the polynomial that determines the critical points of $g_L$:
$$p'(z)\,q(z) - p(z)\,q'(z)= \alpha\,a(z)^5\, b(z)^4\, c(z)^3\, d(z)^3$$
where $\alpha \in \CC{}$ and
\begin{align*}
a(z)=&\ 15\,z^4+30\,z^2-1\\
b(z)=&\ 15\,z^4+1\\
c(z)=&\ 15\,z^4-\biggl(10-8\,\sqrt{\frac{8}{3}}\,i\biggr)\,z^2-1\\
d(z)=&\ 15\,z^4-\biggl(10+8\,\sqrt{\frac{8}{3}}\,i\biggr)\,z^2-1.
\end{align*}
Expressed in the $z$ coordinate on $L$, the $36$, $45$, $60$, and $\obar{60}$-points on $L$ are the four simple roots of $a$, $b$, $c$, and $d$ respectively.  The multiplicity accounts for
$$4\,(5+4+3+3)=60$$
critical points, exactly the number that a degree-$31$ map on \CP{1} has.

As for the claim regarding the Fatou set, let $F$ be a component of $\mathcal{F}_g$.  A nice result in multi-variable complex dynamics states that a strictly critically-finite holomorphic map $f$ on \CP{2} with an irreducible component $A$ in its critical set such that $f(A)=A$ has a Fatou set consisting only of basins of superattracting periodic points.  (\cite{fs}, Theorem~7.8)  We already have that $g$ satisfies these conditions.  Since the only superattracting periodic points under $g$ are the intersections of $45$-lines, $F$ belongs to a basin of one of these fixed points.
\end{proof}

Next, we take up a question of the topology of the Fatou set.

\begin{thm}
The basins of the superattracting fixed points are dense in \CP{2}.
\end{thm}

\begin{proof}
Take $\jmath \in \mathcal{J}_g$ and let $U$ be a neighborhood of $\jmath$.  By a multi-dimensional analogue of Montel's Theorem, $U$ contains critical or pre-critical points, that is,
$$\bigcup_{k=0}^\infty g^{-k}(\XX)\ \bigcap\ U\ \neq\ \emptyset.$$
(See \cite{crit_fin}, Theorem~5.5.)
Consequently, some $g^m(U)$ eventually meets \XX\ and hence a $45$-line $L$.  Since $g|_L$ is critically finite with fixed critical points, its Fatou set $\mathcal{F}_{g|_L} \subset L$ consists of basins of the superattracting points in $L$.  But, these basins are dense---indeed, have full measure---in $L$.  Thus,
$$g^m(U)\ \cap\ \mathcal{F}_{g|_L}\ \neq\ \emptyset$$
so that $g^m(U)$, hence $U$, contains points in the basin of a superattracting fixed point.  (Note that the claim pertaining to the Fatou set of $g$ in Theorem~\ref{thm:fatou} is a corollary to this result.)
\end{proof}

Plotting the basins of attraction on a $45$-line $L$ reveals sixteen distinct basins, one for each intersection of $L$ with a cluster of $45$-lines.  Four clusters of four lines meet $L$ in $36$-points, four clusters of three lines occur at $45$-points, and eight clusters of two lines determine $60$ and $\obar{60}$-points.   Rendering the \CP{1} as a sphere in Figure~\ref{fig:gL45}, each colored region represents a basin for one of the superattracting points.  The prominent regions indicate immediate basins each of which corresponds to one of the points in Figure~\ref{fig:crit_L45} projected onto the sphere.  Although the overall symmetry of $L$ is dihedral ($D_4$), the symmetry \emph{on} the line is Klein-$4$, since reflection through $L$ acts as the identity on $L$.

Due to its $\obar{\V}$ symmetry, $g$ preserves each \R{a}{b}---the \RP{2} mirror fixed point-wise by one of the anti-holomorphic generators of the extended Valentiner group $\obar{V}$.   Figure~\ref{fig:gR36} shows the basins of attraction for eleven superattracting fixed points on the \RP{2}. (Here, the plane is projected onto a disk.)  In these coordinates, $(0,0)$ is the $36$-point to which the \RP{2} corresponds while five $36$-points and five $45$-points are placed symmetrically about the center.  In fact, the $45$-points are on the line at infinity---the boundary circle.  Here, the plane's $D_5$ structure is apparent with the five lines of reflective symmetry occurring where certain $45$-lines meet the real plane in an \RP{1}.  Two of these \RP{1}s show up as lines of reflective symmetry in Figure~\ref{fig:gL45}. 

\begin{figure}[ht]

\resizebox{\textwidth}{!}{\includegraphics{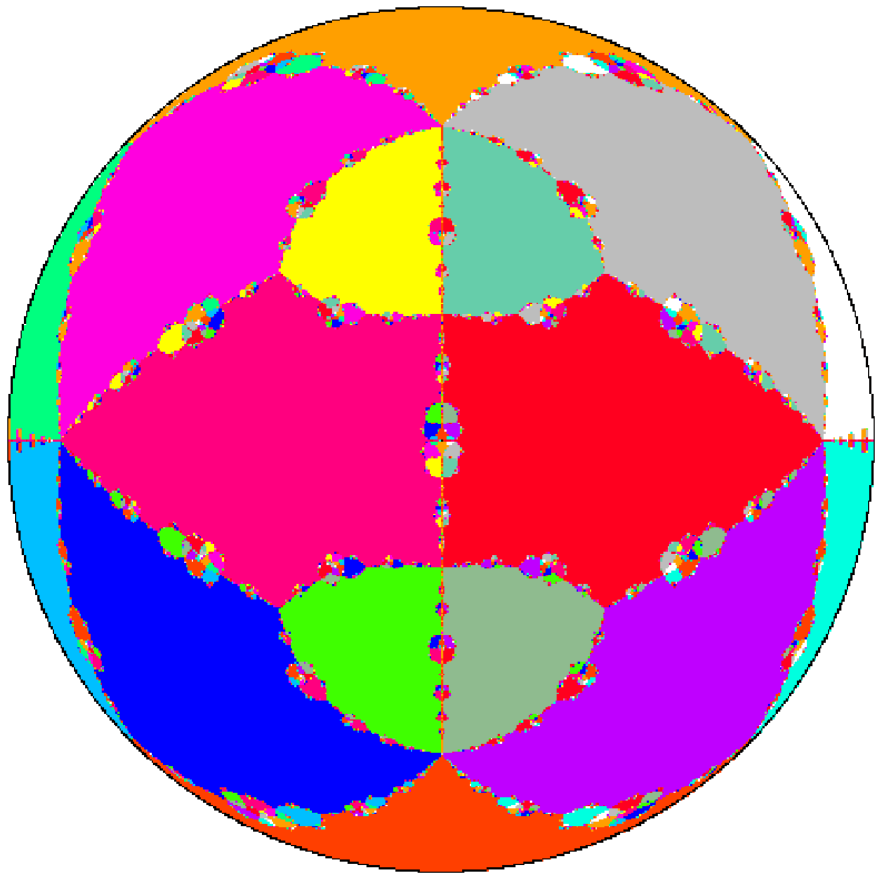}}

\caption{Basins of attraction on a $45$-line.  (Produced by \emph{Julia}.)}

\label{fig:gL45}

\end{figure}

\begin{figure}[ht]

\resizebox{\textwidth}{!}{\includegraphics{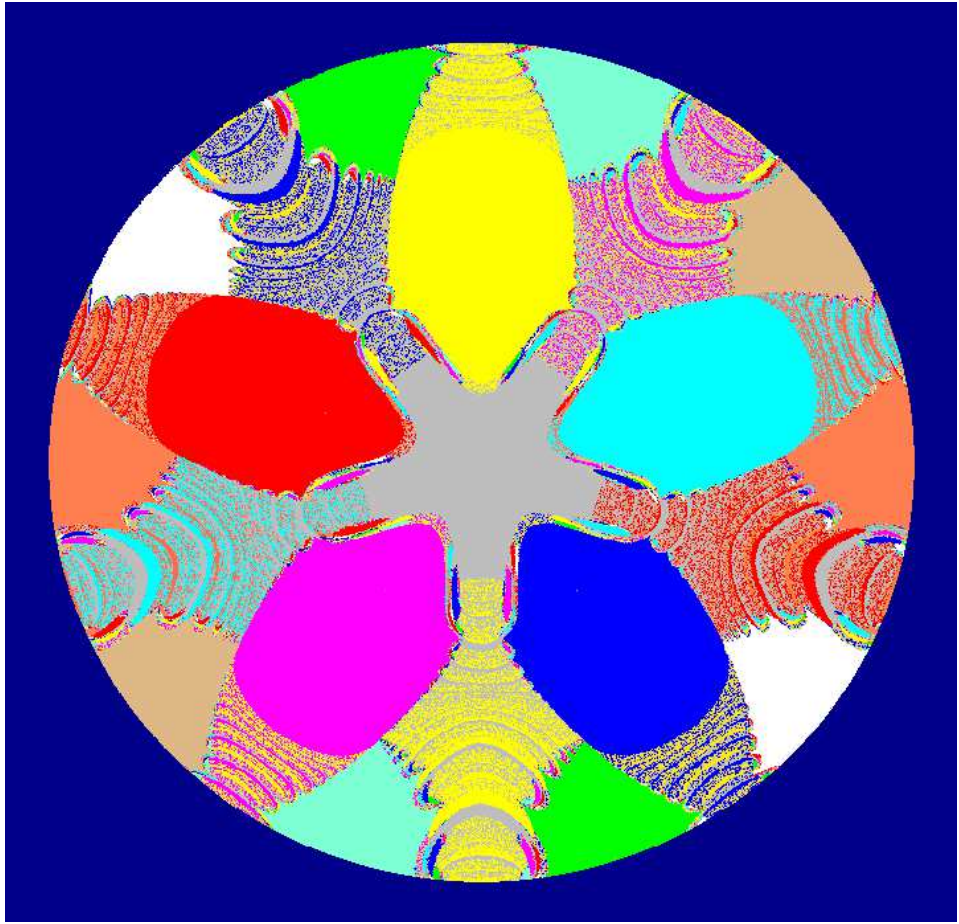}}

\caption{Basins of attraction on a $36$-\RP{2}.  (Produced by \emph{Dynamics 2}.)}

\label{fig:gR36}

\end{figure}
\section{A new map yields a new algorithm}

Although the general form taken by the current sextic-solving procedure is not new, some of its specific features differ from previous methods with respect to practical implementation.  (Again, an extensive description appears in \cite{sextic}.)

\subsection{Sixth-degree resolvent}

Recall the six quadratic forms $C_k$ that one system of icosahedral subgroups in \Vh{3} preserves, up to a factor of a cube root of unity.  By designating the six values
$$s_k(x)=\frac{C_k(x)^3}{F(x)}$$
as roots, the sextic
$$P_x(z) = \prod_{k=1}^6 (z - s_k(x)) = \sum_{k=1}^6 a_k(x)\,z^k$$
has coefficients $a_k(x)$ that are \Vh{6}-invariant rational functions.  Hence, each $a_k$ is expressible in terms of the Valentiner parameters
$$V_1=\frac{\Phi(x)}{F(x)^2} \qquad V_2=\frac{\Psi(x)}{F(x)^5}$$
giving a family of sextic resolvents
 \small \begin{align*}
&P_V(z) =\
z^6+\frac{1}{90} i (5 i+\sqrt{15}) z^5+\frac{((-9-3 i \sqrt{15}) V_1-11 i \sqrt{15}+11) z^4}{24300}\\
&+\frac{(9 (30+i \sqrt{15}) V_1+57 i \sqrt{15}+100) z^3}{12301875}\\
&+\frac{(9 i V_1 (3 (4 i+\sqrt{15}) V_1+8 \sqrt{15}+42 i)-17 i \sqrt{15}-152) z^2}{2214337500}\\
&+\frac{(3 V_1 (9 (-25+33 i \sqrt{15}) V_1+386 i \sqrt{15}+150)-1944 i \sqrt{15} V_2+103 i \sqrt{15}+425) z}{14946778125000}\\
&+\frac{9 V_1 (9 V_1 ((45+11 i \sqrt{15}) V_1-i \sqrt{15}+25)-7 i \sqrt{15}+15)-3 i \sqrt{15}-5}{2421378056250000}.
\end{align*} \normalsize
Hence, each $a_k$ is expressible in terms of the Valentiner parameters
$$V_1=\frac{\Phi(x)}{F(x)^2} \qquad V_2=\frac{\Psi(x)}{F(x)^5}$$
giving a family of sextic resolvents.  After making a simplifying change of parameter, the family is given by
\small \begin{align*}
&P_V(z) =\
-\frac{i (3 \sqrt{15}-5 i) (36\,V_1+1)^3}{2421378056250000}\\
&\frac{(1296 (65-i \sqrt{15})\,V_1^2+144
  (130+21 i \sqrt{15})\,V_1+V_2}{14946778125000}\,z\\
&+\frac{(162
  (19+3 i \sqrt{15})\,V_1^2+(-54+594 i
  \sqrt{15})\,V_1-17 i \sqrt{15}-152)}{2214337500}\,z^2\\
&+\frac{(27 (35-9 i
  \sqrt{15})\,V_1+57 i \sqrt{15}+100)}{12301875}\,z^3\\
&+\frac{(-72\,V_1-11 i
  \sqrt{15}+11)}{24300}\,z^4\\
&+\frac{1}{90} i
  (\sqrt{15}+5 i)\,z^5+z^6
\end{align*} \normalsize 
The crucial observation is that the $31$-map $g$ provides a symmetry-breaking tool for solving members of $P_V$.

\subsection{Parametrizing the Valentiner group}

Define the map
$$y=T_x w = \Psi(x)\, x\,w_1 + \Phi(x)\,h(x)\,w_2 + F(x)\,k(x)\,w_3$$
that is linear in $w=(w_1,w_2,w_3)$ and degree-$31$ in the parameter $x=(x_1,x_2,x_3)$.  Parametrize in $x$ a family of Valentiner groups
$$\V_x= T_x\,\V\,T_x^{-1}$$
where \V\ is the Valentiner group in a selected system of coordinates $y$.  (Recall that $h$ and $k$ are \V-equivariants of degrees $19$ and $25$ respectively.)

Accordingly, an $x$- parametrized version $F_x(w)=F(T_x w)$
of the invariant $F$ arises.  Moreover, each coefficient in $F_x(w)$ of the monomials in $w_k$ is a $\V_x$ invariant of degree $6\cdot 31$ and is, thereby, expressible in terms of the \emph{basic invariants} $F(x)$, $\Phi(x)$, and $\Psi(x)$.  When normalized to a \V-invariant rational function of degree zero, the coefficients are given in  terms of the parameters $V=(V_1,V_2)$:
$$F_V(w) = \frac{F_x(w)}{F(x)^{31}}.$$
Similar considerations yield parametrized forms $\Phi_V(w)$, $\Psi_V(w)$, and $X_V(w)$.

From the $V$-parametrized basic invariants, generating equivariant maps $\psi_V$, $\phi_V$, $f_V$ follow as does the critically-finite $31$-map
$g_V$.  (For an exactly analogous derivation of $g_V$, see \cite{heptic168}.
To this end, the degree-$76$ expression $X\cdot g$ is especially handy.)

\subsection{Finding a $45$-line}

For a random initial point $p_0$ in the $w$-space, the dynamical ``output"
$$g_V^k (p_0) \stackrel{k\rightarrow \infty}{\longrightarrow} p$$
is an intersection of $45$-lines in the \CP{2} under the Valentiner group corresponding to the parameter choice $V$.  At this point we encounter a complication not seen in the original sextic-solving algorithm: the possible values of $p$ are not members of a \emph{single} Valentiner orbit.  Furthermore, the combinatorial properties of the various intersections make for an awkward procedure that employs the ``point datum" $p$.  For instance, a $36$-point (at five $45$-lines) is canonically associated with one conic---hence, one root of $P_V$---in each system, whereas a $45$-point (where four $45$-lines meet) is associated with a pair of conics in each system.  In order to develop a procedure that avoids this pitfall, we can use a dramatic property of the $16$-map $\psi$ whereby each $45$-line collapses to its companion $45$-point.

The idea is to iterate \emph{not} until a fixed point is located, but only until a $45$-line is found.  Recall that each $45$-line is superattracting in the normal direction so that $p_k=g_V^k(p_0)$ will first feel the influence of a line $L$ and subsequently that of a superattracting fixed point on $L$.  Say the trajectory $p_k$ has converged ``linewise" but not pointwise so that
$$L_k=\alpha\,p_{k-1} + \beta\,p_k$$
is nearly a $45$-line.  Application of $\psi_V$ produces a near $45$-point $$a_k = \psi_V(L_k).$$
To achieve high precision we can now compute the trajectory $\{g_V(a_k)^\ell\ |\ \ell=1\dots\}$ to convergence at a $45$-point $a$, a result that provides for two roots of $P_V$.

\subsection{Selecting a pair of roots}

By direct computation, there are \Vh{6}-invariant and degree-$0$ rational functions $A_k(y)$ and $B_{k\ell}(y)$ with the following properties:
\begin{align*}
&A_k(\pV{ac}{bd})=\begin{cases}
1 \quad k=b\ \text{or}\ k=d\\
0 \quad k \neq b\ \text{and}\ k\neq d
\end{cases}\quad 1\leq k < \ell \leq 6 \\
&B_{k\ell}(\pV{ac}{bd})=\begin{cases}
1 \quad k=b\ \text{and}\ \ell=d\\
0 \quad k \neq b\ \text{or}\ \ell\neq d
\end{cases} \quad 1\leq k < \ell \leq 6.
\end{align*}
In the usual way, the functions
$$
A_V(w) = \sum_{k=1}^6 A_k(T_x w)\,s_k(x)\qquad
B_V(w) = \sum_{1\leq k < \ell \leq 6}B_{k\ell}(T_x w)\,s_k\,s_\ell
$$
have $\V_x$-invariant $w$ coefficients and so, can be $V$-parametrized.  Evaluated at a $45$-point
$$p_{45}^w=T^{-1}(\pV{ac}{bd})$$
in the $w$-space,
$$
A=A_V(p_{45}^w) = s_b + s_d \quad \text{and}\quad
B=B_V(p_{45}^w) = s_b\,s_d.
$$
Now, compute the roots $s_b$ and $s_d$ by solving the quadratic equation
$$z^2 - A\,z + B=0.$$

For almost any choice of $V$ and initial point $p_0$, the dynamics of $g_V$ finds a $45$-line---hence, a $45$-point---in the $w$-space.  Accordingly, a pair of roots to $P_V$ is computable from the iterative results.  (A \emph{Mathematica} notebook with supporting files that implements the algorithm is available at \cite{web}.)  
\section*{Acknowledgments}

This work stems from discussions with Peter Doyle.  The referee's comments contributed to a substantial improvement in the article, including the appearance of Figure~\ref{fig:gL45}.

\end{document}